\documentclass[reqno]{amsart}
\usepackage{hyperref}

\begin{document}
\title[\hfil Existence and concentration of solutions]
{Existence and concentration of solutions for a fractional Schr\"odinger equations with sublinear nonlinearity}

\author[J. Zhang; W. Jiang \hfil \hfilneg]
{ Jinguo Zhang and Weifeng Jiang}  

\address{Jinguo Zhang \newline
School of Mathematics,
Jiangxi Normal University,
Nanchang, 330022 Jiangxi, China}
\email{jgzhang@jxnu.edu.cn}

\address{Weifeng Jiang \newline
School of Science,
Wuhan University of Technology,
Wuhan, 430070 Hubei, China}


\begin{abstract}
 This article concerns the fractional elliptic equations
 \begin{equation*}
   (-\Delta)^{s}u+\lambda V(x)u=f(u), \quad u\in H^{s}(\mathbb{R}^N),
 \end{equation*}
 where $(-\Delta)^{s}$ ($s\in (0\,,\,1)$) denotes the fractional Laplacian, $\lambda >0$ is a parameter,
  $V\in C(\mathbb{R}^N)$ and $V^{-1}(0)$
 has nonempty interior. Under some mild assumptions, we establish the existence
 of nontrivial solutions. Moreover, the concentration of solutions is also
 explored on the set $V^{-1}(0)$ as $\lambda\to\infty$.
\end{abstract}

\maketitle
\numberwithin{equation}{section}
\newtheorem{theorem}{Theorem}[section]
\newtheorem{lemma}[theorem]{Lemma}
\allowdisplaybreaks

\newcommand{\R}{\mathbb{R}^{N}}
\newcommand{\HO}{H^{s}_{0,L}(\mathcal{C}_{\Omega})}
\newcommand{\HR}{H^{s}(\mathbb{R}^{N})}
\newcommand{\HHR}{H^{s}_{L}(\mathbb{R}^{N+1}_{+})}

\section{Introduction and statement of main results}
We consider the nonlinear fractional Schr\"odinger equation
\begin{equation}\label{1.1}
 (-\Delta)^{s}u+\lambda V(x)u=f(u), \quad  u\in H^{s}(\mathbb{R}^N),
\end{equation}
 where $(-\Delta)^{s}$ $(0<s<1)$ is the fractional Laplace operator, $\lambda >0$
is a parameter, and $H^{s}(\mathbb{R}^N)$ is the usual fractional Sobolev space
with the norm
$$
\|u\|_{H^{s}}:=\Big(\int_{\mathbb{R}^{N}}(|(-\Delta)^{\frac{s}{2}}u|^{2}+|u|^{2})dx\Big)^{\frac{1}{2}}.
$$

The fractional Schr\"odinger equation is a fundamental equation of fractional quantum mechanics.
It was discovered by Nick Laskin as a result of extending the Feynman path integral,
from the Brownian-like to L\'evy-like euantumn mechanical paths.
Recently, a great attention has been devoted to the fractional and nonlocal operators of elliptic type,
both for their interesting theoretical structure and in view of concrete applications in many fields.
This type of operator has been studied by many authors \cite{ll2007,xt2010,ajly2011,ege2012,fqt2012,sg2013,zl2014} and references therein.

In \cite{fqt2012}, Felmer et al. proved the existence of positive solutions of nonlinear Schr\"odinger equation
involving the fractional Laplacian in $\mathbb{R}^{N}$.
For the whole space $\mathbb{R}^N$ case, the main difficulty of this problem is
the lack of compactness for Sobolev embedding theorem.
To overcome this difficulty, some authors assumed that the potential $V$ satisfies
some additional condition.  Later, the authors
in \cite{sg2013} considered the equation \eqref{1.1} with the critical exponent growth
They proved that the energy functional possess
the property of locally compact. In this paper, we are interested
in the case that the nonlinearity $f$ is sublinear and indefinite.
To our knowledge, few works concerning on this case up to now.
Motivated by the above articles,  we continue to
consider problem \eqref{1.1} with steep well potential and study the existence
of nontrivial solution and concentration results
under some mild assumptions different from those studied previously.
To reduce our statements, we  make the following assumptions for
potential $V$:

\begin{itemize}
\item[$(V_1)$] $V(x)\in C(\mathbb{R}^N)$ and $V(x)\geq 0$ on
$\mathbb{R}^N$;

\item[$(V_2)$]  There exists a constant $b>0$ such that the set
$V_{b}:=\{x\in \mathbb{R}^N|V(x)<b\}$ is nonempty and has finite Lebesgue measure;

\item[$(V_3)$]  $\Omega=\text{int} V^{-1}(0)$ is nonempty and has smooth boundary with
$\bar{\Omega}=V^{-1}(0)$.
\end{itemize}

Based on the above assumptions, the main purpose of this paper is to prove
the existence of nontrivial solutions and to investigate the
concentration phenomenon of solutions on the set $V^{-1}(0)$ as
 $\lambda\to \infty$.
This kind of potential $\lambda V$ satisfying
$(V_1)-(V_3)$ is referred as the steep well potential.
It is worth mentioning that some papers always
assumed the potential $V(x)>0$ for all $x\in \mathbb{R}^{N}$.
Compared with the case $V>0$, our assumptions on $V$ are rather weak,
and perhaps more important.
 To state our results, we need the following assumptions:
\begin{itemize}
\item[($f_1$)]  $f\in C(\mathbb{R}^N, \mathbb{R})$ and there exist constants
 $1<p<2$ and functions
 $\xi(x)\in L^{\frac{2}{2-p}}(\mathbb{R}^N, \mathbb{R}^{+})$ such that
$$
|f(u)|\leq \xi(x)|u|^{p-1},\quad \text{for all}\,\,\,
u\in  \mathbb{R}.
$$

\item[($f_2$)]  There exist three constants $\eta, \delta >0, \gamma\in (1, 2)$
such that
$$
|F(u)|\geq \eta |u|^{\gamma} \quad \text{and all}\,\,\,x\in \Omega\,\,\,\text{and}\,\,\,|u|\leq \delta,
$$
where $F(u)=\int_0^{u}f(t)dt$.
\end{itemize}

On the existence of solutions we have the following result.

 \begin{theorem} \label{thm1.1}
Assume that the conditions $(V_1)-(V_3)$, $(f_1)$ and $(f_2)$ hold.
Then there exists $\Lambda_0>0$ such that
 for every $\lambda>\Lambda_0$, problem \eqref{1.1} has at least one
 solution $u_{\lambda}$.
\end{theorem}

 On the concentration of solutions we have the following result.

 \begin{theorem} \label{thm1.2}
 Let $u_{\lambda}$ be a solution of problem \eqref{1.1} obtained
 in Theorem \ref{thm1.1}, then $u_{\lambda}\to u_0$ strongly in $H^{s}(\mathbb{R}^N)$
 as $\lambda\to \infty$, where $u_0$
 is a nontrivial solution of the equation
 \begin{equation}\label{1.2}
\left\{
\aligned
&(-\Delta)^{s}u=f(u), \quad &x\in\Omega,\\
&u=0,\quad  &x\in \partial\Omega.
\endaligned\right.
\end{equation}
\end{theorem}

The paper is organized as follows.
In Section $2$, we give some preliminary results.
In Section3, we finish the proof of Theorem \ref{thm1.1}.
In Section $4$, we study the concentration of solutions and prove
Theorem \ref{thm1.2}.

\section{Preliminary results}
The fractional Laplacian $(-\Delta)^{s}$ with $s\in (0\,,\,1)$ of a function $u:\, \R\to \mathbb{R}$ is defined by
$$
\mathcal{F}((-\Delta)^{s}u)(\xi) = |\xi|^{2s}\mathcal{F}(u)(\xi), \quad \forall  s \in (0\,,\,1),
$$
where $\mathcal{F}$ is the Fourier transform.

Recently, Caffarelli and Silvestre \cite{ll2007} developed a local interpretation
of the fractional Laplacian given in $\mathbb{R}^{N}$ by considering a Neumann type operator in the extended
domain $\mathbb{R}_{+}^{N+1}$ defined by $\{(x,t)\in  \mathbb{R}^{N+1}:\, t>0\}$.
 For $u\in H^{s}(\mathbb{R}^{N})$, the solution $w\in H^{s}_{L}(\mathbb{R}^{N+1}_{+})$ of
 \begin{equation*}
 \left\{
 \aligned
 -div(t^{1-2s}\nabla w)&=0\quad \text{in}\,\,\,\mathbb{R}^{N+1}_{+},\\
 w&=u\quad \text{in}\,\,\,\mathbb{R}^{N+}\times\{0\},\\
 \endaligned
 \right.
 \end{equation*}
is called $s$-harmonic extension $w=E_{s}(u)$ of $u$ and it is proved in \cite{ll2007}
that
$$
\lim\limits_{t\to 0^+}t^{1-2s}\frac{\partial w}{\partial t}(x,t)=-k_{s}(-\Delta)^{s}u(x),
$$
where
$k_{s}:=2^{1-2s}\Gamma(1-s)\Gamma(s)^{-1}$, the space $H^{s}_{L}(\mathbb{R}^{N+1}_{+})$
is defined as the completion of $C_{0}^{\infty}(\overline{\mathbb{R}^{N+1}_{+}})$
under the norm
$$
\|w\|_{H^{s}_{L}}:=\Big(k_{s}\int_{\mathbb{R}^{N+1}_{+}}t^{1-2s}|\nabla w(x,t)|^{2}dxdt\Big)^{\frac{1}{2}},
$$

 A similar extension, for nonlocal problems on bounded domain $\Omega$ with the zero Dirichlet boundary condition
 was established. In this case, the space $H^{s}_{0,L}(\mathcal{C}_{\Omega})$
is defined as the completion of $C_{0}^{\infty}(\overline{\mathcal{C}}_{\Omega})$
under the norm
$$
\|w\|_{H^{s}_{0,L}}:=\Big(k_{s}\int_{\mathcal{C}_{\Omega}}t^{1-2s}|\nabla w(x,t)|^{2}dxdt\Big)^{\frac{1}{2}},
$$
where $\mathcal{C}_{\Omega}:=\Omega\times(0\,,\,+\infty)\subset \mathbb{R}^{N+1}_{+}$,
some more detail see \cite{bes2013,beau2012}.

In this paper, our problem \eqref{1.1} will be studied in the half-space, namely,
\begin{equation}\label{eq1-1*}
 \left\{
 \aligned
 &-div(t^{1-2s}\nabla w)=0\quad &\text{in}\,\,\,\mathbb{R}^{N+1}_{+},\\
 &-k_{s}\frac{\partial w}{\partial \nu}=-\lambda V(x)w(x,0)+ f(w(x,0))
 \quad &\text{in}\,\,\,\mathbb{R}^{N}\times\{0\},\\
 \endaligned
 \right.
 \end{equation}
where
$$
\frac{\partial w}{\partial \nu}:=\lim\limits_{t\to 0^+}t^{1-2s}\frac{\partial w}{\partial t}(x,t).
$$

Consider the energy functional associated to \eqref{eq1-1*} given by
\begin{equation}\label{eq1-4}
\aligned
J_{\lambda}(w)
&=\frac{k_{s}}{2}\int_{\mathbb{R}_{+}^{N+1}}t^{1-2s}|\nabla w(x,t)|^{2}dxdt
+\frac{\lambda}{2}\int_{\R}V(x)|w(x,0)|^{2}dx\\
&-\int_{\R}F(w(x,0))dx,\\
\endaligned
\end{equation}
which is $C^{1}$ with Cateaus derivative
\begin{equation*}\label{eq1-4}
\aligned
\langle J'_{\lambda}(w),v\rangle&=k_{s}\int_{\mathbb{R}_{+}^{N+1}}t^{1-2s}\nabla w\cdot \nabla vdxdt
+\lambda\int_{\R}V(x)w(x,0)v(x,0)dx\\
&-\int_{\R}f(w(x,0))v(x,0)dx,\\
\endaligned
\end{equation*}
for all $w,\,v\in H^{s}_{L}(\mathbb{R}_{+}^{N+1})$.

By the argument as above, if $w\in \HHR$ is a critical point of $J_{\lambda}$, then $u=Tr(w)\in \HR$ is an
energy or weak solution of problem \eqref{1.1}.
The converse is also right. By the equivalence of these two formulations, we will use both
formulations in the sequel to their best advantage.

For $\lambda>0$. Let
$$
E_{\lambda}=\Big\{w\in \HHR: \int_{\mathbb{R}^{N+1}_{+}}t^{1-2s}|\nabla w|^{2}dxdt+\lambda\int_{\mathbb{R}^{N}}V(x)|w(x,0)|^{2}dx<+\infty\Big\},
$$
be equipped with the norm
$$
\|w\|_{\lambda}=\Big(k_{s}\int_{\mathbb{R}^{N+1}_{+}}t^{1-2s}|\nabla w|^{2}dxdt+\lambda\int_{\mathbb{R}^{N}}V(x)|w(x,0)|^{2}dx\Big)^{1/2}.
$$
It is clear that $E_{\lambda}$ is a Hilbert space, and $\|w\|_{1}\leq \|w\|_{\lambda}$ for all $w\in E_{1}$  with $\lambda \geq 1$.
Moreover, for all $w\in E_{\lambda}$, by using $(V_1)-(V_2)$ and  Sobolev inequality, we have
\begin{equation*}
\aligned
&k_{s}\int_{\mathbb{R}^{N+1}_{+}}t^{1-2s}|\nabla w|^{2}dxdt+\int_{\mathbb{R}^{N}}|w(x,0)|^{2}dx\\
&=k_{s}\int_{\mathbb{R}^{N+1}_{+}}t^{1-2s}|\nabla w|^{2}dxdt+\int_{V_{b}}|w(x,0)|^{2}\,dx
 +\int_{\mathbb{R}^N\backslash V_{b}}|w(x,0)|^{2}\,dx\\
&\leq k_{s}\int_{\mathbb{R}^{N+1}_{+}}t^{1-2s}|\nabla w|^{2}dxdt
 +|V_{b}|^{\frac{2^{*}-2}{2^{*}}}
 \Big(\int_{\mathbb{R}^N}|w(x,0)|^{2^{*}}\,dx\Big)^{2/2^*}
 +\int_{\mathbb{R}^N\backslash V_{b}}|w(x,0)|^{2}\,dx\\
&\leq k_{s}\int_{\mathbb{R}^{N+1}_{+}}t^{1-2s}|\nabla w|^{2}dxdt
 +|V_{b}|^{\frac{2^{*}-2}{2^{*}}}
 \Big(\int_{\mathbb{R}^N}|w(x,0)|^{2^{*}}\,dx\Big)^{2/2^*}\\
&\quad  +\frac{1}{\lambda b}\int_{\mathbb{R}^N\backslash V_{b}}\lambda V|w(x,0)|^{2}\,dx\\
&\leq k_{s}\int_{\mathbb{R}^{N+1}_{+}}t^{1-2s}|\nabla w|^{2}dxdt
 +\frac{|V_{b}|^{\frac{2^{*}-2}{2^{*}}}}{S}
 k_{s}\int_{\mathbb{R}^{N+1}_{+}}t^{1-2s}|\nabla w|^{2}dxdt
 +\frac{1}{\lambda b}\int_{\mathbb{R}^N}\lambda V|w(x,0)|^{2}\,dx\\
&\leq\max\Big\{1,1+\frac{|V_{b}|^{\frac{2^{*}-2}{2^{*}}}}{S},\frac{1}{\lambda b}\Big\}
 k_{s}\int_{\mathbb{R}^{N+1}_{+}}t^{1-2s}|\nabla w|^{2}dxdt+\int_{\mathbb{R}^{N}}\lambda V|w(x,0)|^{2}\,dx\\
&:=c_0 k_{s}\int_{\mathbb{R}^{N+1}_{+}}t^{1-2s}|\nabla w|^{2}dxdt+\int_{\mathbb{R}^{N}}\lambda V|w(x,0)|^{2}\,dx,\\
\endaligned
\end{equation*}
for
 $$\lambda\geq\lambda_0:=\frac{S}{b(S+|V_{b}|^{\frac{2^{*}-2}{2^{*}}})}.
$$
So,
there exist positive constants $\lambda_0$ and $c_0$, independent of $\lambda$, such that
\begin{equation}\label{eq1-2}
\|w\|_{1}\leq  c_0\|w\|_{\lambda},\quad \text{for all }
 u\in E_{\lambda},\; \lambda\geq\lambda_0.
\end{equation}
Furthermore, the embedding $E_{\lambda}\hookrightarrow L^{p}(\mathbb{R}^N)$
is continuous for $p\in[2,2^{*}_{s}]$, and
$E_{\lambda}\hookrightarrow L_{Loc}^{p}(\mathbb{R}^N)$ is compact for
 $p\in[2,2^{*}_{s})$, i.e., there are constants $c_{p}>0$ such that
\begin{equation}\label{2.1}
\|w(x,0)\|_{L^p}\leq c_{p}\|w\|_{1}
\leq c_{p}c_0\|w\|_{\lambda},\quad \text{for all }
u\in E_{\lambda},\; \lambda\geq\lambda_0.
\end{equation}

In order to prove Theorem \ref{thm1.1}, we use the following result by Rabinowitz \cite{Rabinowitz}.
 \begin{lemma} \label{lem2.1}
Let $E$ be a real Banach space and $\Phi \in C^{1}(E, \mathbb{R})$
satisfy the (PS)-condition. If $\Phi$ is bounded from below, then
$c=\inf_{E} \Phi$ is a critical value of $\Phi$.
\end{lemma}
\section{Proof of Theorem \ref{thm1.1}}
In this section, we will finish the proof of Theorem \ref{thm1.1}. First, we give some useful lemmas.
\begin{lemma} \label{lem2.2}
Assume that $(V_1)-(V_3)$, $(f_1)$ and $(f_2)$ hold. Then
there exists $\Lambda_0>0$ such that for every $\lambda>\Lambda_0$,
$J_{\lambda}$ is bounded from below in $E_{\lambda}$.
\end{lemma}

\begin{proof}
From \eqref{2.1}, $(f_1)$ and the H\"older inequality,
we have
\begin{equation}\label{2.4}
\begin{aligned}
 J_{\lambda}(w)&=\frac{1}{2} \|w\|_{\lambda}^{2}-\int_{\mathbb{R}^N}F(w(x,0))\,dx\\
&\geq \frac{1}{2} \|w\|_{\lambda}^{2}-
\Big(\int_{\mathbb{R}^N}|\xi(x)|^{\frac{2}{2-p}}\,dx\Big)
^{(2-p)/2}\Big(\int_{\mathbb{R}^N}|w(x,0)|^{2}\,dx\Big)^{p/2}\\
&\geq \frac{1}{2}
\|w\|_{\lambda}^{2}-c_{2}^{p}
c_0^{p}\|\xi\|_{L^{\frac{2}{2-p}}}\|w\|_{\lambda}^{p},
\end{aligned}
\end{equation}
which implies that $J_{\lambda}(w)\to +\infty$ as
$\|w\|_{\lambda}\to +\infty$, since
$1<p<2$. Consequently, there
exists $\Lambda_0:=\max\{1,\lambda_0\}>0$ such that for every
$\lambda>\Lambda_0$,
$J_{\lambda}$ is bounded from below and coerciveness on $E_{\lambda}$.
\end{proof}

\begin{lemma} \label{lem2.3}
 Suppose that $(V_1)-(V_3)$, $(f_1)$ and $(f_2)$ are satisfied.
Then $J_{\lambda}$ satisfies the (PS)-condition for each
$\lambda>\Lambda_0$.
\end{lemma}

\begin{proof}
Assume that $\{w_n\}\subset E_{\lambda}$ is a sequence such that
$J_{\lambda}(w_n)$ is bounded and $J'_{\lambda}(w_n)\to 0$ as
$n\to \infty$. By Lemma \ref{lem2.2}, it is clear that $\{w_n\}$ is bounded
in $E_{\lambda}$. Thus, there exists a constant $C>0$ such that
\begin{equation}\label{2.5}
\|w_n(x,0)\|_{L^p}\leq c_{p}c_0\|w_n\|_{\lambda}
\leq C,\quad \text{for all } w\in E_{\lambda},\; \lambda\geq\lambda_0,
\end{equation}
where $2\leq p\leq 2^{*}_{s}$.
Passing to a subsequence if necessary, we may assume that
$w_n \rightharpoonup w_0$ weakly in $E_{\lambda}$.
For any $\epsilon >0$, since
$\xi(x)\in L^{\frac{2}{2-p}}(\mathbb{R}^N, \mathbb{R}^{+})$,
we can choose $R_{\epsilon}>0$ such that
\begin{equation}\label{2.6}
\Big(\int_{\mathbb{R}^N\setminus B_{R_{\epsilon}}}|\xi(x)|
^{\frac{2}{2-p}}\,dx\Big)
^{(2-p)/2}<\epsilon.
\end{equation}
From $E_{\lambda}\hookrightarrow L^{p}$ and  $w_n \rightharpoonup _0$
weakly in $E_{\lambda}$, we have
$w_n(x,0)\to w_0(x,0)$ strongly in $L_{ Loc}^{2}( \mathbb{R}^N)$.
Hence
\begin{equation}\label{2.7}
\lim_{n\to \infty}\int_{B_{R_{\epsilon}}} |w_n(x,0)-w_0(x,0)|^{2}\,dx=0.
\end{equation}
Therefore, from  \eqref{2.7}, there exists $N_0\in \mathbb{N}$ such that
\begin{equation}\label{2.8}
\int_{B_{R_{\epsilon}}} |w_n(x,0)-w_0(x,0)|^{2}\,dx<\epsilon ^{2},
\quad \text{for }n\geq N_0.
\end{equation}
Hence, by $(f_1)$, \eqref{2.5}, \eqref{2.8} and the H\"older inequality, for any
$n\geq N_0$, we have
\begin{equation}\label{2.9}
\begin{aligned}
&\int_{B_{R_{\epsilon}}}\left |f(w_n(x,0))-f(w_0(x,0))\right| |w_n(x,0)-w_0(x,0)|\,dx\\
&\leq \Big(\int_{B_{R_{\epsilon}}}|f(w_n(x,0))-f(w_0(x,0))|^{2}\,dx\Big)^{1/2}
\Big(\int_{B_{R_{\epsilon}}}|w_n(x,0)-w_0(x,0)|^{2}\,dx\Big)^{1/2}\\
&\leq \epsilon\Big(\int_{B_{R_{\epsilon}}}2\left(|f(w_n(x,0))|^{2}+|f(w_0(x,0))|^{2}\right)
 \,dx\Big)^{1/2}\\
&\leq 2\epsilon\Big[
\Big(\int_{B_{R_{\epsilon}}}|\xi(x)|^{2}\left(|w_n(x,0)|^{2(p-1)}
+|w_0(x,0)|^{2(p-1)}\right)\,dx\Big)^{1/2}\Big]\\
&\leq 2\epsilon \Big[\|\xi\|_{L^{\frac{2}{2-p}}}^{2}
\Big(\|w_n(x,0)\|_{2}^{2(p-1)}
+\|w_0(x,0)\|_{L^2}^{2(p-1)}\Big)\Big]^{1/2}\\
&\leq 2\epsilon\Big[\|\xi\|_{L^{\frac{2}{2-p}}}^{2}
\Big(C^{2(p-1)}
+\|w_0(x,0)\|_{L^2}^{2(p-1)}\Big)\Big]^{1/2}.
\end{aligned}
\end{equation}
On the other hand, by \eqref{2.5}, \eqref{2.6}, \eqref{2.8} and $(f_1)$, we have
\begin{equation}\label{2.10}
\begin{aligned}
&\int_{\mathbb{R}^N\setminus B_{R_{\epsilon}}}\left |f(w_n(x,0))-f(w_0(x,0))\right| |w_n(x,0)-w_0(x,0)|\,dx\\
&\leq 2\int_{\mathbb{R}^N\setminus B_{R_{\epsilon}}}|\xi(x)|
\left(|w_n(x,0)|^{p}+|w_0(x,0)|^{p}\right)\,dx\\
&\leq 2\epsilon c_{2}^{p}c_0^{p}
\left(\|w_n\|_{\lambda}^{p}+\|w_0\|_{\lambda}^{p}\right)\\
&\leq  2\epsilon c_{2}^{p}c_0^{\gamma_i}
\left(C^{p}+\|w_0\|_{\lambda}^{p}\right).
\end{aligned}
\end{equation}
Since $\epsilon$ is arbitrary, combining \eqref{2.9} with \eqref{2.10}, we have
\begin{equation}\label{2.11}
\int_{\mathbb{R}^N}\left |f(w_n(x,0))-f(w_0(x,0))\right| |w_n(x,0)-w_0(x,0)|\,dx<\epsilon,
\end{equation}
as $n\to \infty$. Hence,
\begin{equation}\label{2.12}
\begin{aligned}
&\langle J'_{\lambda}(w_n)-J'_{\lambda}(w_0), w_n-w_0\rangle=\|w_n-w_0\|_{\lambda}^{2} \\
&+\int_{\mathbb{R}^N}
 (f(
w_n(x,0))-f(w_0(x,0)))(w_n(x,0)-w_0(x,0))\,dx.
\end{aligned}
\end{equation}
From, $\langle J'_{\lambda}(w_n)-J'_{\lambda}(u_0),
w_n-w_0\rangle\to 0$,  \eqref{2.11} and \eqref{2.12},
we get $w_n\to w_0$ strongly in $E_{\lambda}$. Hence, $J_{\lambda}$ satisfies
(PS)-condition.
\end{proof}

\begin{proof}[Proof of Theorem \ref{thm1.1}]
 From Lemmas \ref{lem2.1}, \ref{lem2.2}, \ref{lem2.3}, we  know that
$c_{\lambda}=\inf_{E_{\lambda}}J_{\lambda}(w)$
is a critical value of functional $J_{\lambda}$;
that is, there exists a critical point $w_{\lambda}\in E_{\lambda}$ such that
$J_{\lambda}(w_{\lambda})=c_{\lambda}$.
Next, similar to the argument in \cite{TL}, we show that $w_{\lambda}\neq 0$.
Let $w^{*}\in H_{0,L}^{s}(\mathcal{C}_{\Omega})\setminus \{0\}$
and $\|w^{*}\|_{L^{\infty}}\leq 1$, then by $(f_2)$, we have
\begin{equation}\label{2.13}
\begin{aligned}
 J_{\lambda}(tw^{*})
&=\frac{1}{2} \|tw^{*}\|_{\lambda}^{2}
-\int_{\mathbb{R}^N}F( tw^{*}(x,0))\,dx\\
&= \frac{t^{2}}{2} \|w^{\ast}\|_{\lambda}^{2}-\int_{\Omega}F( tw^{*}(x,0))\,dx\\
&\leq \frac{t^{2}}{2} \|w^{*}\|_{\lambda}^{2} -\eta
t^{\gamma}\int_{\Omega}|w^{*}|^{\gamma}\,dx,
\end{aligned}
\end{equation}
where $0<t<\delta$, $\delta$ be given in $(f_2)$. Since $1<\gamma<2$,
it follows from \eqref{2.13} that
$J_{\lambda}(tw^{\ast})<0$ for $t>0$ small enough. Hence,
$J_{\lambda}(w_{\lambda})=c_{\lambda}<0$, therefore, $w_{\lambda}$ is a
nontrivial critical point of $J_{\lambda}$ and so $w_{\lambda}$
is a nontrivial solution of problem  \eqref{eq1-1*}, that is,
 $u_{\lambda}(x):=Tr(w_{\lambda})=w_{\lambda}(x,0)$
is a nontrivial solution of problem  \eqref{1.1}. The
proof is complete.
\end{proof}

\section{Concentration of solutions}

In the following, we study the concentration of solutions for
problem \eqref{1.1} as $\lambda\to\infty$.
Define
$$
\tilde{c}=\inf_{w\in \HO}J_{\lambda}|_{\HO}(w),
$$
where $J_{\lambda}|_{\HO}$ is a restriction of $J$ on
$ \HO$; that is,
$$
J_{\lambda}|_{\HO}(w)=\frac{k_{s}}{2}
\int_{\mathcal{C}_{\Omega}}t^{1-2s}|\nabla w|^{2}dxdt-\int_{\Omega}F(w(x,0))dx,
$$
for $w\in \HHR$.
Similar to the proof of Theorem \ref{thm1.1}, it is easy to prove that
$\tilde{c}<0$ can be achieved.
Since $\HO\subset E_{\lambda}$ for all $\lambda >0$, we get
$$
c_{\lambda}\leq \tilde{c}<0,\quad \text{for all }\lambda>\Lambda_0.
$$
\begin{proof}[Proof of Theorem \ref{thm1.2}]
 We follow the arguments in \cite{BPW}. For any sequence
$\lambda_n\to \infty$, let $w_n:=w_{\lambda_n}$ be the critical points of
$J_{\lambda_{n}}$ obtained in Theorem \ref{thm1.1}. Thus
\begin{equation}\label{3.1}
J_{\lambda_n}(w_n) \leq \tilde{c}<0
\end{equation}
and
\begin{align*}
 J_{\lambda_n}(w_n)&=\frac{1}{2} \|w_n\|_{\lambda_n}^{2}-
\int_{\mathbb{R}^N}F(w_n(x,0))dx\\
&\geq \frac{1}{2}
\|w_n\|_{\lambda_n}^{2}-c_{2}^{p}c_0^{p}
\|\xi\|_{\frac{2}{2-p}}\|w_n\|_{\lambda_n}^{p},
\end{align*}
which implies
\begin{equation}\label{3.2}
\|w_n\|_{\lambda_n}\leq C,
\end{equation}
where the constant $C>0$ is independent of $\lambda_n$.
Therefore, we may assume that $w_n\rightharpoonup w_0$ in $E_{\lambda}$
and $w_n(x,0)\to w_0(x,0)$ in $L_{\rm loc}^{p}(\mathbb{R}^N)$ for
$2\leq p< 2^{\ast}_{s}$. From Fatou's lemma, we have
$$
\int_{\mathbb{R}^N}V(x)|w_0(x,0)|^{2}\,dx
\leq \liminf_{n\to \infty}\int_{\mathbb{R}^N}V(x)|w_n(x,0)|^{2}\,dx
\leq \liminf_{n\to \infty}\frac{\|w_n\|_{\lambda_n}^{2}}{\lambda_n}=0,
$$
which implies that $w_0=0$ a.e. in $\mathbb{R}^N \setminus
\overline{V^{-1}(0)}$ and $w_0\in \HO$ by
$(V_3)$. Now for any $\varphi \in C_0^{\infty}(\mathcal{C}_{\Omega})$, since
$\langle J'_{\lambda_n}(w_n), \varphi \rangle=0$, it is easy
to verify that
$$
k_{s}\int_{\mathcal{C}_{\Omega}}t^{1-2s}\nabla w_{0}\cdot
\nabla \varphi dxdt-\int_{\Omega}f(w_0(x,0))\varphi \,dx=0,
$$
which implies that $u_{0}(x):=Tr(w_0)$ is a weak solution of equation
\eqref{1.2} by the density of $C_0^{\infty}(\overline{\mathcal{C}}_{\Omega})$ in $ \HO$.

Next, we show that
$w_n(x,0)\to w_0(x,0)$ in $L^{p}(\mathbb{R}^N)$ for $2\leq p<
2_{\ast}$. Otherwise, by Lions vanishing lemma \cite{L,W}, there
exist $\delta>0, \rho>0$ and $(x_n,y)\in \mathbb{R}^{N+1}_{+}$ such that
$$
\int_{B^{+}_{\rho}\cap\{y=0\}}|w_n-w_0|^{2}\,dx\geq \delta,
$$
where $B_{\rho}^{+}:=\{(x,y):\,|(x,y)-(x_{n},y)|<\rho, y>0\}$,
and its base denotes by $B_{\rho}$.
Since $w_n(x,0)\to w_0(x,0)$ in $L_{Loc}^{2}(\mathbb{R}^N)$, let $|x_n|\to \infty$,
we have $|B_{\rho}\cap V_{b}|\to 0$.
 By the H\"older inequality, we get
$$
\int_{B_{\rho}\cap V_{b}}|w_n(x,0)-w_0(x,0)|^{2}\,dx
\leq |B_{\rho}\cap V_{b}|^{\frac{2_{*}-2}{2_{*}}}
 \Big(\int_{\mathbb{R}^N}|u_n-u_0|^{2_{*}}\Big)^{2/2_*}\to 0.
$$
Consequently,
\begin{align*}
\|w_n\|_{\lambda_n}^{2}&\geq \lambda_n b
\int_{B_{\rho}\cap V_{b}^{\perp}}|w_n(x,0)|^{2}dx\\
&=\lambda_nb\Big(\int_{B_{\rho}\cap V_{b}^{\perp}}|w_n(x,0)-w_0(x,0)|^{2}\,dx+
\int_{B_{\rho}\cap V_{b}^{\perp}}|w_0(x,0)|^{2}dx\Big)+o(1)\\\
&\geq \lambda_nb\Big(\int_{B_{\rho}}|w_n(x,0)-w_0(x,0)|^{2}\,dx
-\int_{B_{\rho}\cap V_{b}}|w_n(x,0)-w_0(x,0)|^{2}\,dx\Big)+o(1)\\
&\to \infty\,\,\,\text{as}\,\,n\to \infty,
\end{align*}
which contradicts \eqref{3.2}.  By virtue of
$\langle J'_{\lambda_n}(w_n), w_n \rangle=\langle J'_{\lambda_n}(w_n), w_0 \rangle=0$ and the fact that
$w_n(x,0)\to w_0(x,0)$ strongly in $L^{p}(\mathbb{R}^N)$ for $2\leq p< 2^{*}_{s}$, we have
$$
\lim_{n\to \infty}\|w_n\|_{\lambda_n}^{2}=\|w_0\|^{2}_{\lambda_n}.
$$
Hence, $ w_n\to w_0$ strongly in  $E_{\lambda}$.
Moreover, from \eqref{3.1}, we have $w_0\neq 0$. This completes the proof.
\end{proof}


\begin{thebibliography}{00}
\bibitem{beau2012}
B. Barrios, E. Colorado, A. de Pablo, U. S'anchez,
On some critical problems for the fractional Laplacian,
J. Differential Equations 252 (2012) 6133--6162.

\bibitem{bes2013}
B. Barrios,   E. Colorado, A. de Pablo and  U. S\'{a}nchez, 
 A concave-convex elliptic problem involving the fractional Laplacian.
Proc. Royal Soc. Edi. A 143 (2013) 39--71. 

\bibitem{BPW}
T. Bartsch, A. Pankov, Z. Q. Wang;
Nonlinear Schr\"odinger equations with steep potential wel, Commun.
Contemp. Math. 3 (2001) 549--569.

\bibitem{ll2007}
L. Caffarelli, L. Silvestre, An extension problems related to the fractional Laplacian, Comm. PDE 32 (2007)1245每
1260. 2


\bibitem{xt2010}
X. Cabr∩e, J.G. Tan, Positive solutions of nonlinear problems involving the square root of the Laplacian, Adv. Math.
224 (2010),2052每2093.


\bibitem{ajly2011}
A. Capella, J. D\'avila, L. Dupaigne, Y. Sire, Regularity of radial extremal solutions for some non local semilnear
equations, Comm. PDE 36 (2011) 1353每1384.

\bibitem{ege2012}
E. Di Nezza, G. Palatucci, E. Valdinoci, Hitchhiker＊s guide to the fractional Sobolev spaces, Bull. Sci. Math. 136
(2012) 521每573. 1, 12

\bibitem{fqt2012}
P. Felmer, A. Quaas and J. Tan, Positive solutions of nonlinear Schr\"odinger equation with fractional Laplacian,
Proc. Royal Soc. Edi. {\bf 142}, 1237--1262 (2012)


\bibitem{L} 
P. L. Lions;
The concentration-compactness principle in the calculus of variations.
The local compact case Part I, Ann. Inst. H. Poincar\'{e} Anal. NonLin\'{e}aire,
1 (1984) 109--145.


\bibitem{Rabinowitz} P. H. Rabinowitz;
Minimax methods in critical point theory with
applications to differential equations, CBMS Regional Conf. Ser. in. Math., 65,
American Mathematical Society, Providence, RI, 1986.

\bibitem{sg2013}
Z. Shen and F. Gao, Existence of solutions for a fractional Laplacian equations with critical nonlineatity,
 Abst. Appl. Anal. 2013, Article ID 638425.



\bibitem{TL} X. H. Tang, X. Y. Lin;
Infinitely many homoclinic orbits for Hamiltonian
systems with indefinite sign subquadratic potentials, Nonlinear Anal.
74 (2011) 6314--6325.

\bibitem{W} Y. H. Wei;
Multiplicity results for some fourth-order elliptic equations,
J. Math. Anal. Appl., 385 (2012) 797--807.



\bibitem{zl2014}
J. Zhang and X. Liu, Positive solutions to some asymptotically linear fractional Schr\"odinger equations,
arXiv:1411.2189.


\end{thebibliography}
\end{document}